\documentclass[12pt,reqno,sumlimits]{amsart}
\usepackage{amssymb,amscd,amsmath,epsfig,}
\usepackage[margin=1 in]{geometry}
\newtheorem{thm}{Theorem}
\newtheorem{proposition}{Proposition}
\newtheorem{lemma}{Lemma}

\usepackage{amsthm}
\usepackage{cite}

\newcommand{\R}{\mathbb R}
\newcommand{\emb}{{\text {Emb}}(M,\R^N)}

\begin{document}

\title{Gradient Flows of Penalty Functions in the Space of Smooth Embeddings}
\author{Dara Gold}
\address{Department of Mathematics and Statistics, Boston University}
\email{daragold@bu.edu}

\begin{abstract}

Motivated by manifold learning techniques, we give an explicit lower bound for how far a smoothly embedded 
compact submanifold in $\R^N$ can move in a normal direction and remain an embedding.  In addition, given a penalty function $P : \text{Emb}(M,\mathbb{R}^N) \rightarrow \mathbb{R} $ on the space of embeddings, we give a condition  which guarantees that the gradient $\nabla P$ of the penalty function
is normal to $\phi(M)$ at every point.

\end{abstract} 

\maketitle

\section{Introduction}
In this paper we give some theoretical results and explicit estimates for  gradient flows in infinite dimensions, motivated by issues in manifold learning.  These flows live in the space of embeddings $\emb$ of a fixed $k$-manifold $M$ in $\R^N$.  We give a condition under which the gradient vector field for a functional on $\emb$ is pointwise normal to $\phi(M)$ (where $\phi \in \emb$) (Theorem 1).  There is also a lower bound in Theorem 3 for the existence of a linear flow in a normal direction, which is the type of flow used in numerical simulations.

A fundamental problem in machine learning is to best approximate a fixed finite set of data points in $\mathbb{R}^N$ by a $k$-dimensional embedded manifold. (See {\it e.g.,} \cite{SW} and its references.)  One standard approach to defining an ideal approximation is to impose a penalty function on $\emb$, where the penalty function measures the total distance from the manifold to the data points, the total intrinsic or extrinsic curvature of the embedding, the volume of the embedding, or some linear combination of these terms.  (See \cite{BRS} for the case of fitting data points by a graph $\phi:\R^{N-1}\to \R$.) A minimal embedding for a given penalty function will in general not pass through all data points but will hopefully be more robust with respect to added points.  

In this setup, we have a penalty function $P:\emb\to\R$ on the space of smooth embeddings of a closed $k$-manifold $M$ into $\R^N$.  From a Morse theory perspective, it is natural to study the negative gradient flow of an initial embedding $\phi_0$, with the expectation that for generic $\phi_0$, the long time flow will approach a local minimum of $P$.  

This natural theoretical setup has many technical difficulties. One must choose the topology on 
$\emb$, ideally the Fr\'echet topology on $C^\infty$ embeddings, and check that the penalty function is differentiable enough to have a gradient.  Moreover, the penalty function should contain a distance penalty term, and there are issues of continuity of the gradient vector field at points in $\phi_0(M)$ which
are equidistant from two or more data points.  Most importantly,
it seems very difficult to show  with existing nonlinear PDE techniques that  the short time gradient flow exists for even simple penalty functions. In addition, even if short time, finite time and long time existence are all shown, it is not clear that a limiting point of a 
gradient line is in the original function space.  In our setup, the crucial issue is that $\emb$ is an open subset of ${\rm Maps}(M,\R^N)$ in any reasonable topology, so that it would actually be surprising if the long time flow stays in the space of embeddings.
Finally, because these penalty functions are not convex in general, any limit point need not be a global minimum for the penalty function.  

Despite these problems, the simulations in \cite{BRS} using gradient flow seem robust and better than many other existing manifold learning methods.  Therefore, it seems worthwhile to prove some results that address the theoretical difficulties. 

There is of course a large body of work on gradient flow techniques in infinite dimensions going back to Morse's original papers.  In applied math, there is seminal work of Osher and Sethian \cite{OS}, who introduced the Level Set Method, by which a surface is treated as the level set of a function. This approach, very familiar in finite dimensional Morse theory, avoids typical problems that arise with cusps and discontinuities in a flow whose speed is curvature dependent.  In pure math, 
gradient flows are used to study harmonic map and mean curvature flow/curve shrinking problems; in these cases, the penalty function is an energy or volume functional.  As some examples, 
Rupflin and Topping \cite{RT} study minimal immersions via gradient flow of the harmonic energy map paired with a flow of the Riemannian metric on the domain surface.  For mean curvature flow,
Hamilton \cite{Ham} and Gerhardt \cite{Ger}  prove that convex, compact surfaces in Euclidean space along with curves in a plane contract smoothly to a point under mean curvature flow. Xiao \cite{Xiao} gives a short time estimate for mean curvature flow for immersed star-shaped hypersurfaces in Euclidean space. Huisken and Sinestrari \cite{HS} consider compact hypersurfaces with positive mean curvature to study singularities than can arise during the flow. Using rescaling techniques now familiar in Ricci flow literature, they introduce a 
series of rescaled flows that approach a smooth flow.  Finally, there is a huge body of work on Floer theory, where the energy functional's critical points are pseudoholomorphic curves.

There is some literature which blends theoretical and applied techniques.  Although the gradient of a functional is typically computed using an inner product on the tangent space of the domain space, 
Mayer \cite{May} uses a discretized approximation to the gradient flow, and in particular replaces the time derivative in the 
penalty flow equation with a finite difference term. This leads to a short time existence result that seems so problematic in the smooth setting.  In  \cite{BMM}, there is a detailed discussion of the importance of choosing the right Sobolev space for applications to shape analysis.

In this paper we present first steps towards the  existence of gradient flow in the space of embeddings. In  \S\S1,2, we assume that $M$ is closed. In \S1, we prove that the gradient vector field $Z_{\phi} =\nabla P_{\phi}$ is normal at each point in $\phi(M)$ if and only if $P$ is  invariant under diffeomorphisms of $\phi(M).$ 
In \S2, we consider a fixed normal gradient vector field $\vec{Z}_{\phi}$ along $\phi(M)$ and give an explicit estimate  for how long the flow of embeddings $\phi_t = \phi + tZ$ 
 remains in the space of embeddings. 

\section*{Acknowledgements}
\noindent
Many thanks to Steve Rosenberg for his extensive comments, input and guidance. Also thanks to Carlangelo Liverani for allowing me to use his version of the Quantitative Implicit Function Theorem.

\section{A Condition for Normal Gradient Vector Fields}

In this section, we prove an infinite dimensional analogue of the standard finite dimensional result that gradient vectors are perpendicular to level surfaces.

In the following theorem we use the gradient of the penalty function $\nabla P$, which is defined with respect to the $L^2$ inner product on $T_{\phi}C^{\infty}(M,\mathbb{R}^N)$. For $X \in T_{\phi}C^{\infty}(M,\mathbb{R}^N)$, the gradient is characterized by
$$
dP(X) = \langle \nabla P, X \rangle = \int_{\phi(M)}\nabla P \cdot X \text{dvol},$$
where the volume form is induced from $\mathbb{R}^N$ and we are using the Euclidean dot product. $\nabla P$'s being pointwise normal to $\phi(M)$ means that  $\nabla P_{\phi(m)}\cdot X_{\phi(m)} =0$ for all $\phi(m) \in \phi(M)$.
\begin{thm} For a penalty function $P: C^{\infty}(M,\mathbb{R}^N) \rightarrow \mathbb{R}$, the gradient $\nabla P$ will be normal to $\phi(M)$ for each $m \in M$ if and only if $P$ is invariant under diffeomorphisms $\alpha: \phi(M)\rightarrow \phi(M)$, that are in the path component of the identity in Diff($\phi(M)$), i.e. $P(\alpha(\phi(M))) = P(\phi(M))$.
\end{thm} 
\begin{proof}
$(\Leftarrow)$ Assume $P(\alpha(\phi(M)))= P(\phi(M))$ where $\alpha$ is a diffeomorphism from $\phi(M)$ onto $\phi(M)$ that is generated from the flow of a time independent vector field on $\phi(M)$. We know that $\nabla P_{\phi_0} \perp_{L_2} X_{\phi_0}$ for all $X_{\phi_0}$ that are tangent to the level set containing $\phi_0 \in C^{\infty}(M,\mathbb{R}^N)$.
\smallskip

\noindent \textbf{Claim:} All vector fields $Y_{\phi_0} \in \Gamma(T\phi_0(M))$ lie tangent to the level set of $\phi_0 \in C^{\infty}(M,\mathbb{R}^N)$. 
\medskip

\noindent \textbf{Proof of Claim:} For $Y_{\phi_0} \in \Gamma(T\phi_0(M))$ we have an associated flow along $\phi(M)$ given by $\alpha_{Y,t}:\phi(M) \rightarrow \phi(M)$ and $\dot{\alpha_{Y,t}}\phi(m)= Y_{\phi(m)}$. Furthermore $\alpha_{Y,t}:\phi(M) \rightarrow {\phi(M)}$ is a diffeomorphism for all  $t$. Therefore we can say 
$$
D_{\phi_0}P(Y) =\frac{\text{d}}{\text{dt}}|_{t=0} P(\phi_t) =\frac{\text{d}}{\text{dt}}|_{t=0} P(\alpha_{Y,t}(\phi))=0$$ 
where we have used the assumption and the fact that $\alpha_t(\phi)= \phi_t$. Note: the assumption of diffeomorphism invariance could have been written as:
$$
 P(\phi_t(M))=P(\alpha_t(\phi(M)))=P(\alpha_{t+r}(\phi(M)))$$
 for $t,r \in \mathbb{R}^N$ and $t=0$ in the last equality is the case used above.
\smallskip

We conclude that $\nabla P _{\phi_0} \perp_{L^2} Y_{\phi_0}$ for all vector fields $Y_{\phi_0} \in \Gamma(T\phi_0(M))$.

Fix $\phi(m_0) \in \phi(M)$ and a vector $Q(\phi(m_0)) \in T_{\phi(m_0)}\phi(M)$. 
Choose a sequence of smooth functions $f_{\epsilon_k} : \phi(M)\rightarrow \mathbb{R}$ such that $\int_{\phi(M)}f_{\epsilon_ k}\text{dvol} =1$, $\text{supp} f_{\epsilon_ k} \subset B_{\epsilon_ k}(\phi(m_0)) \cap \phi(M)$ and $\epsilon_k \rightarrow 0$. (Here $B_{\epsilon_ k}(\phi(m_0)) $ is the ball of radius $\epsilon_k$ in $\R^N$ with center at $\phi(m_0).$)
Define  vector fields $Y_{\epsilon_k} $ on $\phi(M)$ by
 $$
 Y_{\epsilon_k}(\phi(m)) = f_{\epsilon_k}(\phi(m))\cdot Q(\phi(m_0)). $$
Then we have 
\begin{eqnarray*}
0&=& 
\underset{\epsilon_k \rightarrow 0}{\text{lim}} \langle \nabla P _{\phi}, Y_{\epsilon_k} \rangle =\underset{\epsilon_k \rightarrow 0}{\text{lim}} \langle \nabla P _{\phi}, f_{\epsilon_k}\cdot Q(\phi(m_0)) \rangle\\
& =& \underset{\epsilon_k \rightarrow 0}{\text{lim}} \int_{\phi(M)} \nabla P _{\phi}(\phi(m))\cdot f_{\epsilon_k} Q(\phi(m_0))   
= \nabla P_ {\phi}{\phi(m_0)} \cdot Q(\phi(m_0)).
\end{eqnarray*}
Therefore $ \nabla P_ {\phi} \perp Y_{\phi}$ pointwise.\\

\noindent
($\Rightarrow$) Assume that $\nabla P_{\phi(m)} \perp \phi(M)$ for all $\phi(m) \in \phi(M)$. This is equivalent to saying $\nabla P_{\phi(m)} \perp Y_{\phi(m)}$ at each point $\phi(m) \in \phi(M)$ for all vector fields $Y \in \Gamma(T\phi(M))$. This gives 
$$
\frac{\text{d}}{\text{dt}}|_{t=0}P(\phi_t) =0,\ \ 
 \dot{\phi_t}|_{t=0} = Y, $$
which means that moving in the direction of the flow $\alpha_{Y,T}$ generated by a fixed vector field $Y$ is equivalent to moving along a level set in $C^{\infty}(M,\mathbb{R}^N)$. Because flows generated in this way are diffeomorphisms from $\phi(M)$ to $\phi(M)$ we can conclude that 
$$
P(\alpha_{Y,t}(\phi(M))) = P(\phi(M)) $$
for all $\alpha,t,Y$. \\
\end{proof}

\section{An Estimate for Flows in Normal Gradient Directions} 
The above result gives a condition for determining if the gradient vector field generated by a penalty function is normal at every point in $\phi(M)$. In the case where this is true, we would next like to consider how far $\phi(M)$ can move in a fixed normal gradient direction while remaining an embedding. The next set of results gives an explicit estimate for the lower bound of this flow. \\ 
\\
\noindent
3.1 \textbf{Notation and Definitions} 
\\
\\
\noindent
\textbullet $\: \:$ $\epsilon$ is the size of the neighborhood around $\phi(M)$ in which each point has a unique closest point in $\phi(M)$. The existence of this neighborhood for $M$ closed is guaranteed by the $\epsilon$-Neighborhood Theorem \cite[Ch.~2, \S3]{GP}. It is given explicitly in Lemma 3 in the proof of Theorem 3. \\
\textbullet $\: \:$ We will use two sets of coordinates on $\mathbb{R}^N$. Standard coordinates will be denoted $(x^1, \dots, x^N)$. We will also be representing points in $\phi(M)$ and in a small neighborhood around $\phi(M)$ as elements of the normal bundle $N\phi(M)$. In coordinates they will be given as $(q^1, \dots, q^k, r^1, \dots, r^{N-k})$ where the first $k$ components are manifold coordinates and the last $N-k$ are coordinates for the normal space. These will be referred to as normal coordinates. For $q \in \phi(M)$, its representation is $(q^1, \dots, q^k, 0, \dots, 0)$. For $w = (q^1, \dots, q^k, r^1, \dots, r^{N-k})$ inside a small neighborhood of $\phi(M)$, $q =(q^1 \cdots q^k, 0 \cdots 0)$ is $w$'s closest point in $\phi(M)$ and $(0,\cdots, 0, r^1, \cdots r^{N-k}) = w-q \in N\phi(M)$. \\
\textbullet $\:\:$ A vector in $N_q\phi(M)$ will be denoted as either $t\vec{v}(q)$ (where $\vec{v}$ is unit length) or as $r^iw_i(q)$ where the $\{w_i \}$ vectors are a unit length spanning set of the normal space at $q$. There are $N-k$ $\{w_i \}$ vectors, each with $N$ coordinates.\\
\textbullet $\: \:$ For $\phi(M) \subset \mathbb{R}^N$, the map $E: N\phi(M) \rightarrow \mathbb{R}^N$ acts by $E(q,r) =q+r$ (sending points to the end of perpendicular vectors in the normal bundle over $\phi(M)$).
 It is given explicitly by:
 \begin{eqnarray*}
E(( q^1, \cdots, q^k, r^1, \cdots, r^{n-k})) &=& (x^1(q)+ r^iw_i^1(q), \cdots, x^N(q)+ r^iw_i^N(q))\\
&=& (\phi^1(q)+ r^iw_i^1(q), \cdots, \phi^N(q)+ r^iw_i^N(q)),
\end{eqnarray*}
 where the domain is in normal coordinates and the range is in standard coordinates. Points $e = q_e+v_e$ for which the Jacobian of the $E$ map isn't full rank (at the point $(q_e,v_e)$) are defined as 'focal points.' \cite{Mil}\\
\textbullet $\: \:$  The inclusion map $\phi(M) \rightarrow \mathbb{R}^N$ takes points $(q^1, \cdots, q^k) \mapsto (x^1(\vec{u}),\cdots, x^N(\vec{u}))$.
 It is a standard result that the first fundamental form is the matrix with entries $(g_{ij}) = \big(\frac{\partial \vec{x}}{\partial u^i} \cdot \frac{\partial \vec{x}}{\partial u^j}\big) $ (Euclidean dot product) and the second fundamental form is the matrix with entries $(\vec{v} \cdot \vec{l}_{ij})$ where $\vec{l}_{ij}$ is the normal component of the vector $\frac{\partial ^2 \vec{x}}{\partial u^i \partial u^j}$.

\noindent
\textbullet $\: \:$ In choosing coordinates that make the first fundamental form the identity matrix, the eiqenvalues $p_1, \cdots, p_k$ of the second fundamental form are called the `principal curvatures' at $q =\phi(m) \in \phi(M)$. Considering the normal line $l= q +t\vec{v}$ extending from $q \in \phi(M)$ ($\vec{v}$ is a fixed unit normal vector at $q$) we have the proposition \cite[p. 34]{Mil}:\\
\begin{proposition}The focal points of $([\phi(M)],q)$ along $l$ are precisely the points $q + p_i^{-1}\vec{v}$, where $ 1 \leq i \leq k, p_i \neq 0$ 
\end{proposition}
\noindent
\textbullet $\: \:$ $K= {\max\limits_{\phi(m) \in \phi(M)}} {p_{\phi(m)}}$ where $p_{\phi(m)}$ is the largest eigenvalue of $(\vec{v}_{\phi(m)} \cdot l_{ij} )$ evaluated at $q=\phi(m) \in \phi(M)$. \\
\textbullet $\: \: $ $\delta$ is chosen such that for $d_{\mathbb{R}^N}(x,y) < \delta/2$ ($x,y \in \phi(M)$) we know $x+ \vec{tv(x)} \neq y+ \vec{tv(y)}$ for $t \leq \delta$. It is defined explicitly after the proof of Lemma 3. 

\medskip
\noindent
\textbf{Note}: The next two theorems are stated in terms of unit length normal vector fields on $\phi(M)$. 
The Euler class of the normal bundle is the obstruction to the existence of such a vector field.
If this class is nonzero, we apply the theorem to vector fields where each vector has length at most one.

\begin{thm}
Let $\vec{v}$ be a normal vector field of length at most one along $\phi(M) \subset \mathbb{R} ^{N}$ and $\epsilon$ be as defined above. $\phi_t(M) = \{{\phi(m) + t\vec{v} : m \in M} \}$ is immersed in $\mathbb{R}^{N}$ for $t < \epsilon$.
\end{thm}
\noindent 
\begin{proof}
We want to show that the map $M \rightarrow \phi_t(M)$ is an immersion for $t$ defined in the theorem statement, but because $\phi(M)$ is assumed to be embedded in $\mathbb{R}^N$ it suffices to show that the map $F:\phi(M) \rightarrow \phi_t(M)$ (where for $q \in \phi(M)$, $F(q) = q +\vec{tv}(q)$) is an immersion. We want to consider $\phi_t(M)$ as sitting in an open subset of $\mathbb{R}^N$ that we can identify with the normal bundle over $\phi(M)$. In particular, the $\epsilon$  - Neighborhood Theorem \cite{GP} gives that on a compact, boundaryless manifold in $\mathbb{R}^N$ -$\phi(M)$ in our case- there exists a sufficiently small $\epsilon$ such that for each point $w$ in $Y^{\epsilon}$-- the set of points in $\mathbb{R}^N$ a distance less that $\epsilon$ from the manifold-- there is a unique closest point $q$ in $\phi(M)$. Furthermore $w-q \in N_q(\phi(M))$ where $N \phi(M)$ is the normal bundle over $\phi(M)$. We can diffeomorphically identify (locally) points in $Y^{\epsilon}$ with elements in $N \phi(M)$ as follows:
$$
w \mapsto (w-q)_q
$$
where $q$ is $w$'s unique closest point in $\phi(M)$. When considering the case of our fixed vector field $t\vec{V}$ along $\phi(M)$ as a section of the normal bundle we get the following coordinate representation of this section:
$$
\phi(m) + \vec{tv}(\phi(m)) \mapsto (q^1, \cdots, q^k, tv^1(q), \cdots, tv^{N-k}(q))
$$
where now the vector components are function of $q$.
Therefore the map:
$$
F: \phi(M) \rightarrow \phi_t(M) \subset Y^{\epsilon}
$$
has the normal coordinate representation:
$$
(q^1, \cdots, q^k) \mapsto (q^1, \cdots, q^k, tv^1(q), \cdots, tv^{N-k}(q))
$$
the differential of which is given by:
$$DF(q)
= \left( \begin{array} {ccc}
\frac{\partial q^1(q)}{\partial q^1} & \cdots & \frac{\partial q^1(q)}{\partial q^k}\\
\vdots & & \vdots\\
\frac{\partial (tv)^{n-k}(q)}{\partial q^1} & \cdots & \frac{\partial (tv)^{n-k}(q)}{\partial q^k} \end{array} \right)
=
\left( \begin{array} {ccc}
1 & \cdots & 0\\
\vdots & & \vdots\\
0 & \cdots & 1\\ 
\vdots & & \vdots\\
\frac{\partial (tv)^{n-k}(q)}{\partial q^1} & \cdots & \frac{\partial (tv)^{n-k}(q)}{\partial q^k} \end{array} \right)
$$
which has rank $k$, showing that the map taking $\phi(M) \rightarrow \phi_t(M)$ is an immersion for $t < \epsilon$. 
\end{proof}

Next, we would like to show that $\phi_t$ is injective, which along with its being an immersion (Theorem 2) and the assumption that $M$ is compact is enough to conclude that $\phi_t$ is an embedding. While Theorem 2 showed that $\phi_t$ is an immersion for $t\leq \epsilon$, Theorem 3 will show injectivity for $t\leq t^*$. Lemma 2 (included in the proof of Theorem 3) shows that $t^* \leq \epsilon$. Therefore the final theorem showing $\phi_t$ is an embedding is on the interval $t \leq t^*$.

The statement of Theorem 3 uses the new value $\delta$ which is defined explicitly after the proof of Lemma 3. Recall that $\delta$ is chosen such that for $d_{\mathbb{R}^N}(x,y) < \delta$ ($x,y \in \phi(M)$) we know $x+ \vec{tv(x)} \neq y+ \vec{tv(y)}$ for $t \leq \delta$.

\begin{thm}
Let $\vec{v}$ be a normal vector field of length at most one along $\phi(M) \subset \mathbb{R} ^{N}$ Let $t^* = \text{min}\{K^{-1}, \delta/3\}$. Then $\phi_t: M \rightarrow \mathbb{R}^N$ given by $m \mapsto \phi(m) + \vec{tv(\phi(m))}$ is 
an embedding for $t \leq t^*$.
\end{thm}
\noindent
\begin{proof}
 It should be noted that we are interested in the injectivity of the map $\phi_t: M \rightarrow \mathbb{R}^N$ defined above, but because $\phi(M)$ is embedded in $\mathbb{R}^N$ it suffices to show that $F: \phi(M) \rightarrow \phi_t(M)$ is injective for $t \leq t^*$.

To view $F$ as a map acting on open subsets of $\mathbb{R}^N$ we define the function $H_t$ from $Y^{\epsilon -t} \rightarrow Y^{\epsilon}$, the set of points a distance $\epsilon-t$ and $\epsilon$  from $\phi(M)$ in $\mathbb{R}^N$ respectively. Setting $\pi: Y^{\epsilon} \rightarrow \phi(M)$ with $\pi(w)$ the closest point in $\phi(M)$ to $w$ we can define:
$$
 H_t(w) = w + \vec{tv}_{\pi(w)}.
 $$
Note that $H_t|_{\phi(M)} =F$.

We continue the proof with a series of Lemmas.

\begin{lemma}
 $DH_t(q_0)$ is invertible for $ w=q_0 \in \phi(M)$ 
 \end{lemma}
 \noindent
 \begin{proof}
 For $H_t: Y^{\epsilon-t} \rightarrow Y^{\epsilon} $ via $w \mapsto w + \vec{tv}(\pi(w))$ its normal coordinate representation (explained in proof of Theorem 1) is given by:
 $$
 (q^1, \cdots, q^k, r^1, \cdots, r^{N-k}) \mapsto (q^1, \cdots, q^k, r^1 + tv^1(\pi(q)), \cdots,r^{n-k} + tv^{N-k}(\pi(q)) )
 $$
 where it should be noted that the $r^i$'s are independent of coordinates but the $v^i(q)$'s are the coordinates for the fixed vector field along $\phi(M)$ which depend on $q$. For $w = q_0 \in \phi(M)$ the differential of the $H_t$ map (taken in coordinates) is given by:
 \begin{eqnarray*} \lefteqn{DH_t(w) }\\
 &=& \left( \begin{array} {cccccc}
 \frac{\partial q^1 (\vec{q},0)}{\partial q^1} & \cdots & \frac{\partial q^1 (\vec{q},0)}{\partial q^k} & \frac{\partial q^1 (\vec{q},0)}{\partial r^1}& \cdots &\frac{\partial q^1 (\vec{q},0)}{\partial r^{n-k}} \\
 \vdots & & & & & \vdots\\
 \frac{\partial q^k (\vec{q},0)}{\partial q^1} & \cdots & \frac{\partial q^k (\vec{q},0)}{\partial q^k} & \frac{\partial q^k (\vec{q},0)}{\partial r^1}& \cdots &\frac{\partial q^k (\vec{q},0)}{\partial r^{n-k}} \\
 \frac{\partial (r^1+ tv^1(q))(\vec{q},0)}{\partial q^1} & \cdots & \frac{\partial (r^1+ tv^1(q))(\vec{q},0)}{\partial q^k} & \frac{\partial (r^1+ tv^1 (q))(\vec{q},0)}{\partial r^1}& \cdots &\frac{\partial (r^1+ tv^1  (q))(\vec{q},0)}{\partial r^{n-k}}\\
 \vdots & & & & & \vdots\\
 \frac{\partial (r^{n-k}+ tv^{n-k}(q))(\vec{q},0)}{\partial q^1} & \cdots & \frac{\partial (r^{n-k}+ tv^{n-k}(q))(\vec{q},0)}{\partial q^k} & \frac{\partial (r^{n-k}+ tv^{n-k} (q))(\vec{q},0)}{\partial r^1}& \cdots &\frac{\partial (r^{n-k}+ tv^{n-k}  (q))(\vec{q},0)}{\partial r^{n-k}} \end{array} \right) \\
&=& \left( \begin{array} {cccccc}
 1 & \cdots & 0 & 0 & \cdots & 0 \\
 \vdots & & & & & \vdots\\
 0 & \cdots & 1 & 0 & \cdots & 0 \\
  \frac{\partial (tv^1(q))(\vec{q},0)}{\partial q^1} & \cdots & \frac{\partial ( tv^1(q))(\vec{q},0)}{\partial q^k} & 1 & \cdots & 0\\
 \vdots & & & & & \vdots\\
  \frac{\partial (tv^{n-k}(q))(\vec{q},0)}{\partial q^1} & \cdots & \frac{\partial ( tv^{n-k}(q))(\vec{q},0)}{\partial q^k} & 0 & \cdots & 1 \end{array} \right) 
  \end{eqnarray*}
       This matrix is invertible for all $t$ so we can conclude that there exists a ball $B_{\delta^{q_0}_{H_t}}$ of radius $\delta_{H_t}^{q_0}$ around $q_0$, on which $H_t$ is a diffeomorphism. 
   \end{proof}

   Let $\delta_{H_t} = \underset{q_0}{\text{min}} \: \: \delta_{H_t}^{q_0}$.
  Although $DH_t$ is invertible for all time (the size of the neighborhood will change according to $t$), we must have $t<\epsilon$ for $H_t$ to be defined. Therefore $t$ is  less than $\epsilon$ and we can say: For $x,y \in \phi(M)$ with $d_{\mathbb{R}^N}(x,y) < \delta_{H_t}$, we have $x+ \vec{tv}(x) \neq y+ \vec {tv}(y)$ for $t < \epsilon$,
  and we can show injectivity:

  \begin{lemma} $H_t|_{\phi(M)}$ is injective for $t < t^*= \text{min}\{\epsilon, \frac{\delta_{H_t}}{3} \}$.
 \end{lemma}
 \begin{proof}
 Assume instead that there exists some $x, y \in \phi(M)$ such that $x + \vec{tv(x)} = y + \vec{tv(y)}$ and $t < t^*$. We know by assumption that $d_{\mathbb{R}^N}(x,y)> \delta_{H_t}$.
  Therefore:
  \begin{eqnarray*}
    \delta_{H_t} < d_{\mathbb{R}^N}(x,y) &=&|x-y|\\
    &=& |x - (x + \vec{tv}(x)) + (x + \vec{tv}(x)) -y| \\
    &=& |x - (x + \vec{tv}(x)) +  (y + \vec{tv}(y)) -y| \\
    & \leq& |x - (x + \vec{tv}(x))| +  |(y + \vec{tv}(y)) -y| \\
     &=& |\vec{tv}(x)| +  |\vec{tv}(y)| = 2|t| < 2|t^*| \\
     &\leq& 2\delta_{H_t}/3 \space \space \end{eqnarray*}
    which is a contradiction. 
    \end{proof}

We now must compute $\epsilon$ (the size of the neighborhood around $\phi(M)$ within which each point has a unique closest point in $\phi(M)$). Lemma 3 again uses $\delta$ which is defined explicitly following the proof. Recall: $\delta$ is chosen such that for $d_{\mathbb{R}^N}(x,y) < \delta$ ($x,y \in \phi(M)$) we know $x+ \vec{tv(x)} \neq y+ \vec{tv(y)}$ for $t \leq \delta$. In the statement of Lemma 3, $\vec{tv}$ has been written in terms of unit length spanning vectors, $w_i$'s of the normal bundle $N\phi(M)$ with coefficients $r^i$ ($1\leq i \leq N-k$).   

\begin{lemma} $\epsilon = \text{min} \{K^{-1}, \delta/3 \}$ where $\delta$ is such that for $x,y \in \phi(M)$ and $d_{\mathbb{R}^N}(x,y) < \delta$ we have $x+r_x^iw_i(x) \neq y+r_y^iw_i(y)$ where $|r_x|<\delta  $ and $|r_y|<\delta  $.  \\
\end{lemma}
 \begin{proof} Suppose there exists $w \in Y^{\epsilon}$ such that there are two closest points $x, y \in \phi(M)$. Then we can write $w= x+r_x^iw_i(x) = y+r_y^iw_i(y)$ where $|r_x|<\epsilon$ and $|r_y|<\epsilon$. We know by assumption that $d_{\mathbb{R}^N}(x,y) > \delta$ and we have a similar proof as in Lemma 2:
\begin{eqnarray*}
\delta < d_{\mathbb{R}^N}(x,y) &=&|x-y|\\
&=&|x-(x+r^i_xw_i(x)) + (x + r^i_xw_i(x)) -y| \\
    & =& |x - (x + r^i_xw_i(x)) +  (y + r^i_yw_i(y)) -y|\\
      &\leq& |x - (x + r^i_xw_i(x))| +  |(y + r^i_yw_i(y)) -y|\\
      &=& |r^i_xw_i(x)| +  |r^i_yw_i(y)| \\ &=&  |r_x|+ |r_y| < 2\epsilon \leq 2\delta/3 
     \end{eqnarray*} 
     which is a contradiction.
     \end{proof}

 We will obtain $\delta$ in the following way: Recall $E: N\phi(M) \rightarrow \mathbb{R}^N$ acts on points in the normal bundle over $\phi(M)$ by $(q,r) \mapsto q+r$. Here we will be considering the compact subset of $N\phi(M)$ which consists of vectors $\vec{r}$ such that $|r| \leq .999K^{-1}$. In coordinates, recall $E$ is given by:
  \begin{eqnarray*}E(( q^1, \cdots, q^k, r^1, \cdots, r^{n-k})) &=& (x^1(q)+ r^iw_i^1(q), \cdots, x^N(q)+ r^iw_i^N(q))\\
 &=& (\phi^1(q)+ r^iw_i^1(q), \cdots, \phi^N(q)+ r^iw_i^N(q)).
 \end{eqnarray*}
 
   Fix  $q_0=(q_0^1, \cdots, q^k_0,0, \cdots, 0) \in \phi(M).$ For a point $(q_0,r_0)$ in the fiber over $q_0$ we know that $DE(q_0,r_0)$ is invertible (see proof of Proposition 2) and therefore there is a ball of radius $\delta_{(q_0,r_0)}$ around $(q_0,r_0)$ on which $E$ is a diffeomorphism. Because the fiber over $q_0$ is compact, we can let
  $\delta_{q_0} = \underset{r_0}{\text{min}} \: \: \delta_{(q_0,r_0)} > 0$. 
  
  Consider the set 
  $$A_{q_0}=\{ q \in \phi(M) \: \: : d_{\mathbb{R}^N} \: (q, q_0) < \delta_{q_0} /2 \}.$$
   Then  $E$ is a diffeomorphism on the subset of $N\phi(M)$ given in normal coordinates by $B_{q_0} = \{ (q^1, \cdots, q^k, r^1, \cdots, r^{n-k}) | \: \: |r|< \delta_{q_0} /2, (q^1, \cdots, q^k, 0, \cdots, 0) \in A_{q_0} \}$ as follows:
  For $(q_1, r_1)\in B_{q_0}$:
  \begin{eqnarray*}
  |(q_1, r_1) - (q_0,0)| &=& |(q_1, r_1) -(q_1,0) +(q_1,0)- (q_0,0)|\\
  &<&   |(q_1, r_1) -(q_1,0)| +|(q_1,0)- (q_0,0)| \\
&=&  |r_1| +|(q_1,0)- (q_0,0)| \\
&<& \delta_{q_0} /2 +\delta_{q_0} /2 =\delta_{q_0}.
\end{eqnarray*}
  Therefore for $(q_1, r_1), (q_2, r_2) \in B_{q_0}$ $\big((q_1, r_1)\neq (q_2, r_2)\big)$ we know $(q_1,0), (q_2,0) \in A_{q_0}$ and $E((q_1,r_1)) = q_1 + r_1^iw_i \neq q_2 +r_2^1w_1 =E ((q_2,r_2))$.
  
  We let
  \begin{equation*}
   \delta =\underset{q_0} {\text{inf}} \: \: \delta_{q_0}/2.
   \end{equation*}
  We can now say that for  $x,y \in \phi(M)$ and $d_{\mathbb{R}^N}(x,y) < \delta$ we have $x+r_x^iw_i(x) \neq y+r_y^iw_i(y)$ for $|r_x|< \delta$ and $|r_y|< \delta$ by construction.\\

  It remains to compute $\delta_{(q_0,r_0)}$ explicitly, from which we can get $\delta$ with the method described above (Recall, $\delta_{(q_0,r_0)}$ is the radius around $(q_0,r_0)$  on which $E$ is a diffeomorphism). We will compute $\delta_{(q_0,r_0)}$ using a quantitative version of the Implicit Function Theorem (adapted to the Inverse Function Theorem case), given as a proposition below. The formulation of the theorem, along with its proof is in the Appendix.

For $G \in C^1(\mathbb{R}^{2N}, \mathbb{R}^N)$, let $(q_0, y_0) \in \mathbb{R}^{2N} $ satisfy $G(q_0, y_0) =0 $. For fixed $\gamma > 0$ let $V_{\gamma} = \{(q,y) \in \mathbb{R}^{2N}: |q-q_0| \leq \gamma, |y-y_0| \leq \gamma \}$.
In the case where $G(q,y) = E(q) - y$, the following theorem is the adaptation of the Implicit Function Theorem to the Inverse Function Theorem (here the matrix norm $||A||$ is the sup norm over the entries):

\begin{proposition} Assume that $\partial_q G(q_0,y_0)$ is invertible and choose $\delta^0 > 0$ such that\\ $\text{sup}_{(q,y)\in V_{\delta^0}}|| 1 - [\partial_q G(q_0,y_0)]^{-1}\partial_q G(q,y)|| \leq  1/2.$ Let $B_{\delta^0} = \text{sup}_{(q,y)\in V_{\delta^0}}||\partial_y G(q,y)||$ and $M= ||\partial_q G(q_0,y_0)^{-1}||$. Let $ \delta_1 = (2MB_{\delta^0})^{-1}\delta^0$ and $\Gamma_{\delta_1} = \{y \in \mathbb{R}^m : ||y-y_0|| < \delta_1 \}.$ Then in the case that $G(q,y) =E(q) -y$, the solutions to $G(q,y)=0 (\Rightarrow E(q)= y)$ in the set $\{ (q,y) : ||q-q_0|| < \delta^0, ||y-y_0|| < \delta_1 \}$ are given by $(E^{-1}(y), y)$. Alternatively, E is a diffeomorphism on $E^{-1}(B_{\delta_1}(y_0)) \cap B_{\delta^0}(q_0)$. 
\end{proposition}

We will apply the proposition to $E: N\phi(M) \rightarrow \mathbb{R}^N$. Specifically, in applying the proposition we have $((q_0, r_0), y_0)$ as a base point (as opposed to simply writing $(q,y)$ as in the proposition statement, we will write $((q,r),y)$ to  emphasize use of normal coordinates), we have $ G((q,r),y) = E(q,r) -y$ and $G((q_0,r_0),y_0) =0 \: (\Rightarrow E((q_0, r_0)) = y_0)$. Therefore:
$$\partial_{(q,r)}G((q_0,r_0),y_0)=DE(q_0,r_0) =
\left( \begin{array} {ccc}
 \frac{\partial \phi^1(q_0,r_0)}{\partial q^1}+r^i\frac{\partial w_i^1(q_0,r_0)}{\partial q^1} & \cdots &  w_{n-k}^1(q_0) \\
 \vdots & & \vdots \\
 \frac{\partial \phi^N(q_0,r_0)}{\partial q^1}+r^i\frac{\partial w_i^N(q_0,r_0)}{\partial q^1} & \cdots &  w_{n-k}^N(q_0) \end{array} \right)$$
which is invertible for $|r|< K^{-1}$ as required by the proposition's assumption. Again, our goal is to get a $\delta_{(q_0,r_0)}$ neighborhood around $(q_0,r_0)$ on which $E$ is a diffeomorphism. Following the proposition's steps we have: 
\smallskip

\noindent 
\textbf{ Step 1:}\\
\begin{eqnarray*}
B_{\delta^0_{(q_0,r_0)}} &=& \text{sup}_{((q,r),y)\in V_{\delta^0_{(q_0,r_0)}}}||\partial_y G((q,r),y)||\\
&=& \text{sup}_{((q,r),y)\in V_{\delta^0_{(q_0,r_0)}}}||\partial_y (E(q,r)-y)||\\
&=& \text{sup}_{((q,r),y)\in V_{\delta^0_{(q_0,r_0)}}} \left \Vert  \left( \begin{array} {ccc}
-1 & 0 & 0 \\
\vdots & & \vdots\\
0& & -1 \end{array}  \right) \right \Vert =1,
\end{eqnarray*}
where we have taken the maximum of the absolute values of the matrix's entries for the matrix norm.\\
\\
\noindent\textbf{Step 2:}\\
$$M= ||\partial_{(q,r)} G((q_0,r_0),y_0)^{-1}||
=||DE(q_0,r_0)^{-1}||.$$ 
 Using Cramer's rule and the matrix adjugate to invert $DE(q_0,r_0)$, we have
$$
  (DE(q_0,r_0)^{-1})_{(j,z)}=  \frac{1}{\text{det}(DE (q_0,r_0))}(-1)^{(z+j)}DE(q_0,r_0)^*_{(j,z)}$$
      where $DE(q_0,r_0)^*_{(j,z)}$ is the $(j,z)$th minor of $DE(q_0,r_0)$, or the determinant of the $(n-1) \times (n-1)$ matrix constructed by deleting the $j$th row and $z$th column of $DE(q_0,r_0)$, which gives an explicit way to compute $M$ above.\\
      \\
      \noindent 
      \textbf{Step 3:} \\

    We want to compute  $\delta^0_{(q_0,r_0)}$ such that $\text{sup}_{((q,r),y)\in V_{\delta^0_{(q_0,r_0)}}}|| 1 - [ DE(q_0,r_0)]^{-1} DE(q,r)|| \leq  1/2$. Since this expression doesn't rely on $y$, we need $\delta^0_{(q_0,r_0)}$ such that for $|(q,r)|< \delta^0_{(q_0,r_0)} \Rightarrow || 1 - [DE(q_0,r_0)]^{-1}DE(q,r)|| \leq  1/2 $. To do this we can consider a first order Taylor series expansion on $DE(q,r)$ around $(q_0, r_0)$. (Note: the $j$ index in the second matrix below refers to coordinates in $\mathbb{R}^N$, not an exponent.) We have:
\begin{eqnarray*}
\lefteqn{ DE(q,r) }\\
&=&\left( \begin{array} {ccc}
     \frac{\partial \phi^1(q_0,r_0)}{\partial q^1}+r^i\frac{\partial w_i^1(q_0,r_0)}{\partial q^1} & \cdots &  w_{n-K}^1(q_0) \\
     \vdots & & \vdots \\
     \frac{\partial \phi^N(q_0,r_0)}{\partial q^1}+r^i\frac{\partial w_i^N(q_0,r_0)}{\partial q^1} & \cdots &  w_{n-K}^N(q_0) \end{array} \right)\\
&&\quad  + \left( \begin{array} {ccc}
      \sum\limits_{j=1}^{N}R^{(1,1)}_j (q,r)(z-z_o)^j & \cdots &\sum\limits_{j=1}^{N}R^{(1,N)}_j (q,r)(z-z_o)^j \\
      \vdots & & \vdots \\
       \sum\limits_{j=1}^{N}R^{(N,1)}_j (q,r)(z-z_o)^j & \cdots &\sum\limits_{j=1}^{N}R^{(N,N)}_j (q,r)(z-z_o)^j  \end{array} \right) \\
&=& \left( \begin{array} {ccc}
            \frac{\partial \phi^1(q_0)}{\partial q^1}+r_0^i\frac{\partial w_i^1(q_0)}{\partial q^1} & \cdots &  w_{n-K}^1(q_0) \\
            \vdots & & \vdots \\
            \frac{\partial \phi^N(q_0)}{\partial q^1}+r_0^i\frac{\partial w_i^N(q_0)}{\partial q^1} & \cdots &  w_{n-K}^N(q_0) \end{array} \right)\\
&&\quad  + \left( \begin{array} {ccc}
             \sum\limits_{j=1}^{N}R^{(1,1)}_j (q,r)(z-z_o)^j & \cdots &\sum\limits_{j=1}^{N}R^{(1,N)}_j (q,r)(z-z_o)^j \\
             \vdots & & \vdots \\
              \sum\limits_{j=1}^{N}R^{(N,1)}_j (q,r)(z-z_o)^j & \cdots &\sum\limits_{j=1}^{N}R^{(N,N)}_j (q,r)(z-z_o)^j  \end{array} \right) 
              \end{eqnarray*}
  where $z-z_o = (q^1-q_0^1, \cdots, q^k-q_0^k, r^1-r_0^1, \cdots, r^{n-k}-r_0^{N-k})$. We have a uniform bound on the error term given by:
  $$
  |R_j^{(l,m)}(q,r)|\leq \text{max} \big \{\left|\frac{\partial f^m_l((q,r))}{\partial z^{j}}\right| : 1 \leq j \leq N, r \leq .999 K^{-1}, q \in \phi(M) \big \} \overset {\text{def}}{=} G^{(m,l)}  $$
  For $(l,m)$ with $1 \leq l \leq N$ and $1\leq m \leq k$, $f^m_l= \frac{\partial \phi^m(q)}{\partial q^l}+r^i\frac{\partial w_i^m(q)}{\partial q^l}$. For 
  $(l,m)$ with $1 \leq l \leq N$ and $k+1\leq m \leq N$, $f^m_l= w^l_m(q)$.\\

 Plugging the above sum for $DE(q,r)$ in the expression $|| 1 - [DE(q_0,r_0)]^{-1}DE(q,r)||    $ we see that the first term cancels with the identity matrix and we are left with:
 
 $$ \left \Vert [DE(q_0,r_0)]^{-1} \left( \begin{array} {ccc}
       \sum\limits_{j=1}^{N}R^{(1,1)}_j (q,r)(z-z_o)^j & \cdots &\sum\limits_{j=1}^{N}R^{(1,N)}_j (q,r)(z-z_o)^j \\
       \vdots & & \vdots \\
        \sum\limits_{j=1}^{N}R^{(N,1)}_j (q,r)(z-z_o)^j & \cdots &\sum\limits_{j=1}^{N}R^{(N,N)}_j (q,r)(z-z_o)^j  \end{array} \right) \right \Vert $$

$$= \left \Vert \left( \begin{array} {ccc}
  ([DE(q_0,r_0)]^{-1})_{(1,p)}\sum\limits_{j=1}^{N}R^{(p,1)}_j (q,r)(z-z_o)^j & \cdots & ([DE(q_0,r_0)]^{-1})_{(1,p)}\sum\limits_{j=1}^{N}R^{(p,N)}_j (q,r)(z-z_o)^j\\
  \vdots & & \vdots\\
  ([DE(q_0,r_0)]^{-1})_{(N,p)}\sum\limits_{j=1}^{N}R^{(p,1)}_j (q,r)(z-z_o)^j & \cdots & ([DE(q_0,r_0)]^{-1})_{(N,p)}\sum\limits_{j=1}^{N}R^{(p,N)}_j (q,r)(z-z_o)^j \end{array} \right) \right \Vert $$
  $$ \leq \left \Vert \left( \begin{array} {ccc}
    ([DE(q_0,r_0)]^{-1})_{(1,p)}\delta^0_{(q_0,r_0)}\sum\limits_{j=1}^{N}R^{(p,1)}_j (q,r) & \cdots & ([DE(q_0,r_0)]^{-1})_{(1,p)}\delta^0_{(q_0,r_0)}\sum\limits_{j=1}^{N}R^{(p,N)}_j (q,r)\\
    \vdots & & \vdots\\
    ([DE(q_0,r_0)]^{-1})_{(N,p)}\delta^0_{(q_0,r_0)}\sum\limits_{j=1}^{N}R^{(p,1)}_j (q,r) & \cdots & ([DE(q_0,r_0)]^{-1})_{(N,p)}\delta^0_{(q_0,r_0)}\sum\limits_{j=1}^{N}R^{(p,N)}_j (q,r) \end{array} \right) \right \Vert $$
    \begin{equation}\label{two} \leq \left \Vert \left( \begin{array} {ccc}
        ([DE(q_0,r_0)]^{-1})_{(1,p)}\delta^0_{(q_0,r_0)}NG^{(p,1)} & \cdots & ([DE(q_0,r_0)]^{-1})_{(1,p)}\delta^0_{(q_0,r_0)}NG^{(p,N)}\\
        \vdots & & \vdots\\
        ([DE(q_0,r_0)]^{-1})_{(N,p)}\delta^0_{(q_0,r_0)}NG^{(p,1)} & \cdots & ([DE(q_0,r_0)]^{-1})_{(N,p)}\delta^0_{(q_0,r_0)}NG^{(p,N)} \end{array} \right) \right \Vert 
        \end{equation} 

        Letting  $\delta^0_{(q_0,r_0)} =\frac{1} {2\underset{(l,m)}{\text{max}}([DE(q_0,r_0)]^{-1})_{(l,p)}NG^{(p,m)}}$ we have that the last term in (1) does not exceed 1/2, as each entry has absolute value less than $1/2$ by construction. 
\smallskip

        \noindent
        \textbf{ Step 4:} 
        \smallskip
        
        Now that we have a value for $\delta^0_{(q_0,r_0)}$ we can compute $\delta^1_{(q_0,r_0)}$ as in the proposition statement by:
        $$
        \delta^1_{(q_0,r_0)} = (2MB_{\delta^0_{(q_0,r_0)}})^{-1}\delta^0_{(q_0,r_0)} =(2M)^{-1}\delta^0_{(q_0,r_0)} $$
        where the last equality is from Step 1 and $M$ is computed in Step 2.
        \smallskip

        \noindent\textbf{Step 5:}
        \smallskip
        
        By the theorem statement we know $E$ is a diffeomorphism on 
        $$P_{(q_0,r_0)}=E^{-1}(B_{\delta^1_{(q_0,r_0)}}(y_0)) \cap B_{\delta^0_{(q_0,r_0)}}(q_0,r_0).$$ In particular we need a ball of radius $\delta_{(q_0,r_0)}$ around $(q_0,r_0)$ on which $E$ is a diffeomorphism. First, we need a $\delta^3_{(q_0,r_0)}$ such that for $|(q,r)-(q_0,r_0)|<\delta^3_{(q_0,r_0)} $ implies  $|E (q,r) - E(q_0,r_0)| = |E (q,r) - y_0| < \delta^1_{(q_0,r_0)}$. We can again compute this $\delta^3_{(q_0,r_0)}$ using a Taylor series expansion of $E$ around $(q_0,r_0)$. We have
        $$
        E(q,r) = E(q_0,r_0)+ \big( \sum\limits_j R^1_j(q,r)((q,r)-(q_0,r_0))^j, \cdots, \sum\limits_j R^N_j(q,r)((q,r)-(q_0,r_0))^j \big)
        $$
        where we have bounds on the error terms given by:
        $$
        |R^p_j(q,r)| \leq \text{max}\left \{ \left|\frac{\partial (\phi^p+ r^iw_i^p)(q,r)}{\partial z^{j}} \right| : 1 \leq j \leq N, q \in \phi(M), r \leq .999 K^{-1} \right\} \overset{\text{def}}{=} G^{p}.  $$
           Then we have
\begin{eqnarray*}
          \lefteqn{|E(q,r) -E(q_0,r_0)|^2}\\
           &=& |\big(\sum\limits_j R^1_j(q,r)((q,r)-(q_0,r_0))^j, \cdots, \sum\limits_j R^N_j(q,r)((q,r)-(q_0,r_0))^j\big)|^2\\
          &=& \sum\limits_{p=1}^N (\sum\limits_jR^p_j(q,r)((q,r)-(q_0,r_0))^j)^2
         = \sum\limits_{p=1}^N |\sum\limits_jR^p_j(q,r)((q,r)-(q_0,r_0))^j|^2\\
& \leq&  \sum\limits_{p=1}^N \sum\limits_j|R^p_j(q,r)((q,r)-(q_0,r_0))^j|^2 \leq \sum\limits_{p=1}^N \sum\limits_j|G^p((q,r)-(q_0,r_0))^j|^2\\
           &\leq & \sum\limits_{p=1}^N \sum\limits_j|G^p \delta^3_{(q_0,r_0)}|^2  = (\delta^3_{(q_0,r_0)})^2\sum\limits_{p=1}^N \sum\limits_j|G^p |^2 = N(\delta^3_{(q_0,r_0)})^2\sum\limits_{p=1}^N |G^p |^2.
           \end{eqnarray*}
Therefore
          $$  |E(q,r) -E(q_0,r_0)| \leq \delta^3_{(q_0,r_0)}\sqrt{N\sum\limits_{p=1}^N |G^p |^2},$$
          and letting $\delta^3_{(q_0,r_0)} = \delta^1_{(q_0,r_0)}/ \left (\sqrt{N\sum\limits_{p=1}^N |G^p |^2} \right )$ gives the required radius. We finally set $\delta_{(q_0,r_0)} = \text{min} \{\delta^3_{(q_0,r_0)}, \delta^0_{(q_0,r_0)}\}$

        Finally, returning to the statement in Lemma 2, we had:
        $H_t|_{\phi(M)}$ is injective for $t < t^*= \text{min}\{\epsilon, \frac{\delta_{H_t}}{3} \}= \text{min}\{K^{-1}, \delta/3, \frac{\delta_{H_t}}{3} \}$. By definition we know that for $x,y \in \phi(M)$ and $d_{\mathbb{R}^N}(x,y) < \delta$ we have $x+r_x^iw_i(x) \neq y+r_y^iw_i(y)$ (where $|r_x|<\delta $ and $|r_x|<\delta$). However, we also have that for $x,y \in \phi(M)$ satisfying $d_{\mathbb{R}^N}(x,y) < \delta_{H_t}$, $x+ \vec{tv}(x) \neq y+ \vec {tv}(y)$ for $t < \epsilon < \delta$. Therefore we can say that $\delta< \delta_{H_t}$. This is because our specific vector field $t\vec{v}$ gives a particular set of $r^i$'s at each point, allowing for a larger diffeormorphic neighborhood around the base point than a neighborhood that works for all set of $r^i$'s. Therefore we have
 $H_t|_{\phi(M)}$ is injective for $t < t^*= \text{min}\{\epsilon, \frac{\delta_{H_t}}{3} \}= \text{min}\{K^{-1}, \frac{\delta}{3}, \frac{\delta_{H_t}}{3} \} = \text{min}\{K^{-1}, \frac{\delta}{3} \}$ as required. 

 We have shown that $\phi_t$ is an injective immersion for $t\leq t^*$ (by the fact that $t^* \leq \epsilon$). Since $M$ is compact $\phi_t$ is an embedding.  

   This concludes the proof of Theorem 3.
        \end{proof}



\section*[appendix]{Appendix: The Quantitative Implicit Function}

\noindent
This quantitative version of the Implicit Function theorem and its variation of standard proof techniques is due to Calangelo Liverani \cite{CL}.

Fix $n,m  \in \mathbb{N}$ and $F \in C^1(\mathbb{R}^{n+m}, \mathbb{R}^m)$ and let $(x_0, \lambda_0) \in \mathbb{R}^m \times \mathbb{R}^n$ satisfy $F(x_0, \lambda_0) =0 $. For $\delta > 0$ let $V_{\delta} = \{ (x,\lambda) \in \mathbb{R}^{m+n} : || x-x_0|| \leq \delta, ||\lambda -\lambda_0|| \leq \delta \}$. 

\begin{thm} {\rm (Quantitative Implicit Function Theorem)}
Assume that $\partial_xF(x_0,\lambda_0)$ is invertible and choose $\delta > 0$ such that $\text{sup}_{(x,\lambda)\in V_{\delta}}||1 - [\partial_xF(x_0,\lambda_0)]^{-1}\partial_xF(x,\lambda)|| \leq 1/2$. Let $B_{\delta} = \text{sup}_{(x,\lambda)\in V_{\delta}}|| \partial_{\lambda}F(x,\lambda)||$ and $M = ||\partial_xF(x_0,\lambda_0)^{-1}||$. Set $\delta_1 = (2MB_{\delta})^{-1}\delta$ and $\Gamma_{\delta_1} = \{ \lambda \in \mathbb{R}^n : \: \: || \lambda - \lambda_0|| < \delta_1 \}$. Then there exists $g \in C^1(\Gamma_{\delta_1}, \mathbb{R}^m)$ such that all the solutions of the equation $F(x, \lambda) =0$ in the set $\{ (x, \lambda) : \: \: || \lambda - \lambda_0|| < \delta_1, ||x-x_0|| < \delta\}$ are given by $(g(\lambda), \lambda)$. In addition,
  $ \partial_{\lambda}g(\lambda) = - (\partial_x F(g(\lambda), \lambda))^{-1} \partial_{\lambda} F(g(\lambda), \lambda)$ 
\end{thm}
\textbf{Proof:} Set $A(x, \lambda) = \partial_x F(x, \lambda), M = ||A(x_0, \lambda_0)^{-1}||$. 

We want to solve the equation $F(x, \lambda)=0$. Let $\lambda$ be such that $ ||\lambda - \lambda_0|| < \delta_1 \leq \delta$. Consider $U_{\delta} = \{x \in \mathbb{R}^m : ||x-x_0|| \leq \delta \}$ and the function $\Omega : U_{\delta} \rightarrow \mathbb{R}^m$ defined by
$$
 \Omega_{\lambda}(x) = x- A(x_0,\lambda_0)^{-1}F(x, \lambda).
 $$
For $x \in U(\lambda), F(x, \lambda) =0$ is equivalent to $x = \Omega_{\lambda} (x)$. \\
Next, 
$$
  ||\Omega_{\lambda}(x_0) - \Omega_{\lambda_0}(x_0)|| \leq M || F (x_0, \lambda)|| \leq MB_{\delta}\delta_1 
$$
In addition, $||\partial_x \Omega_{\lambda}|| =|| 1 - A(x_0, \lambda_0)^{-1}A(x, \lambda)|| \leq 1/2$. Thus
$$
||\Omega_{\lambda}(x) -x_0|| \leq \frac{1}{2} ||x-x_0|| + || \Omega_{\lambda}(x_0) -x_0|| \leq \frac{1}{2} ||x-x_0|| + MB_{\delta}\delta_1 \leq \delta $$
The existence of $x \in U_{\delta}$ such that $\Omega_{\lambda} (x) =x$ follows by the Fixed Point Theorem. We have therefore obtained a function $g: \Gamma_{\delta_1} = \{ \lambda \: : \: ||\lambda - \lambda_0 || \leq \delta_1\} \rightarrow \mathbb{R}^m$ such that $F(g(\lambda),\lambda) =0$. 

It remains to prove regularity. 
Let $\lambda, \lambda' \in \Gamma_{\delta_1}$. From above we have
$$
||g(\lambda) - g(\lambda') || \leq \frac{1}{2} ||g(\lambda) - g(\lambda') || + M B_{\delta} |\lambda - \lambda '| $$
This yields the Lipschitz continuity of the function $g$. To obtain the differentiability we note that, by the differentiability of $F$ and the above Lipschitz continuity of $g$, for $h \in \mathbb{R}^n$ small enough, 
$$
||F(g(\lambda +h), \lambda +h) -F(g(\lambda), \lambda) + \partial_x F[g(\lambda+h) -g(\lambda), h] + \partial_{\lambda} F(g(h),h) || = o(||h||)
$$
Since $F(g(\lambda+h), \lambda +h) = F(g(\lambda), \lambda) =0$ we have
$$\underset{h \rightarrow 0}{\text{lim}} ||h||^{-1} ||g(\lambda+h) - g(\lambda) + [\partial_xF(g(h),h)]^{-1} \partial _{\lambda}F(g(h),h)|| =0, $$
which concludes the proof.

\bibliography{references}
\bibliographystyle{amsplain}
   \end{document}